\documentclass[reqno,11pt]{amsart}
\usepackage{amssymb}
\usepackage{verbatim}
\newif\ifpdf
\ifpdf
  \usepackage[pdftex]{graphicx}
  \usepackage[pdftex]{hyperref}
\else
  \usepackage{graphicx}
\fi
\numberwithin{equation}{section}       
 \theoremstyle{plain}    
 \newtheorem{thm}{Theorem}[section]
 \numberwithin{equation}{section} 
 \numberwithin{figure}{section} 
 \theoremstyle{plain}
 \theoremstyle{plain}    
 \newtheorem{cor}[thm]{Corollary} 
 \theoremstyle{plain}    
 \newtheorem{prop}[thm]{Proposition} 
 \theoremstyle{plain}    
 \newtheorem{lem}[thm]{Lemma} 
 \theoremstyle{remark}
 \newtheorem{rem}[thm]{Remark}
 \theoremstyle{definition}
 
\theoremstyle{definition}
\newtheorem{defi}[thm]{Definition}
\newtheorem*{thmA}{Theorem A} 
\newtheorem*{thmB}{Theorem B} 
\newtheorem*{thmC}{Theorem C} 

\newtheorem*{corA}{Corollary A}

\newtheorem*{corC}{Corollary C}
\newtheorem*{corD}{Corollary D}

\newtheorem*{ackn}{Acknowledgement}

\newcommand{\C}{{\mathbb{C}}}

\newcommand{\Q}{{\mathbb{Q}}}
\newcommand{\R}{{\mathbb{R}}}

\newcommand{\PP}{{\mathbb{P}}}

\newcommand{\cB}{{\mathcal{B}}}

\newcommand{\cD}{{\mathcal{D}}}
\newcommand{\cE}{{\mathcal{E}}}

\newcommand{\cL}{{\mathcal{L}}}

\newcommand{\cO}{{\mathcal{O}}}

\newcommand{\cT}{{\mathcal{T}}}

\newcommand{\cP}{{\mathcal{P}}}

\renewcommand{\b}{\beta}

\newcommand{\de}{\delta}
\newcommand{\e}{\varepsilon}

\newcommand{\MA}{\mathrm{MA}\,}

\newcommand{\vol}{\operatorname{vol}}

\newcommand{\supp}{\operatorname{supp}}

\newcommand{\Herm}{\operatorname{Herm}}

\newcommand{\ev}{\mathrm{ev}}
\newcommand{\eq}{{\mu_\mathrm{eq}}}
\newcommand{\eneq}{{\cE_\mathrm{eq}}}

\begin{document}

\setcounter{tocdepth}{1}

\title[Fekete points]{Fekete points and convergence towards equilibrium measures on complex
manifolds}

\date{\today{}}

\author{Robert Berman, S{\'e}bastien Boucksom, David Witt Nystr{\"o}m}

\address{Chalmers University of Technology and the University of G{\"o}teborg\\
Department of Mathematics
 SE-412 96 G{\"o}teborg\\
 Sweden}

\email{robertb@chalmers.se}

\address{CNRS-Universit{\'e} Paris 7\\
 Institut de Math{\'e}matiques\\
 F-75251 Paris Cedex 05\\
 France}

\email{boucksom@math.jussieu.fr}

\address{Chalmers University of Technology and the University of G{\"o}teborg\\
Department of Mathematics
 SE-412 96 G{\"o}teborg\\
 Sweden}

\email{wittnyst@chalmers.se}

\begin{abstract}
Building on \cite{BB08a}, we prove a general criterion for convergence
of (possibly singular) Bergman measures towards equilibrium measures
on complex manifolds. The criterion may be formulated in terms of
growth properties of balls of holomorphic sections, or equivalently
as an asymptotic minimization of generalized Donaldson $L$-functionals.
Our result yields in particular the proof of a well-known
conjecture in pluripotential theory concerning the equidistribution
of Fekete points, and it also gives the convergence of Bergman measures
towards equilibrium for Bernstein-Markov measures. The present paper
therefore supersedes our preprints \cite{BB08b,BWN08}. Applications
to interpolation of holomorphic sections are also discussed. 
\end{abstract}

\maketitle

\tableofcontents{}

\section*{Introduction}

\subsection{The setting}

Let $L$ be a holomorphic line bundle over a compact complex manifold
$X$ of complex dimension $n$. Following \cite{BB08a}, let $(K,\phi)$
be a \emph{weighted compact subset}, i.e. a non-pluripolar compact
subset $K$ of $X$ together with the weight $\phi$ of a continuous
Hermitian metric $e^{-\phi}$ on the restriction $L|_{K}$. Finally
let $\mu$ be a probability measure on $K$.

The asymptotic study as $k\rightarrow\infty$ of the space of global
sections $s\in H^{0}(X,kL)$ endowed with either the $L^{2}$ norm
$$
\Vert s\Vert_{L^{2}(\mu,k\phi)}^{2}:=\int_{X}|s|^{2}e^{-2k\phi}d\mu
$$
 or the $L^{\infty}$ norm 
 $$
\Vert s\Vert_{L^{\infty}(K,k\phi)}:=\sup_{K}|s|e^{-k\phi}
$$
 is a natural generalization of the classical theory of orthogonal
polynomials. The latter indeed corresponds to the case 
$$
K\subset\C^{n}\subset\PP^{n}=:X
$$
equipped with the tautological ample bundle $\cO(1)=:L$. It is of
course well-known that $H^{0}(\PP^{n},\cO(k))$ identifies with the
space of polynomials on $\C^{n}$ of total degree at most $k$. The
section of $L$ cutting out the hyperplane at infinity induces a flat
Hermitian metric on $L$ over $\C^{n}$, so that a continuous weight
$\phi$ on $L|_{K}$ is naturally identified with a function in $C^{0}(K)$.
On the other hand, a psh function on $\C^{n}$ with at most logarithmic
growth at infinity gets identified with the weight $\phi$ of a non-negatively
curved (singular) Hermitian metric on $L$, which will thus be referred
to as a \emph{psh weight}.

Our geometric setting is therefore seen to be a natural (and more
symmetric) extension of so-called \emph{weighted potential theory}
in the classical case (cf. \cite{ST97} and in particular Bloom's
appendix therein). It also contains the case of \emph{spherical
polynomials} on the round sphere $S^{n}\subset\R^{n+1}$, as studied
e.g.~in \cite{Mar07,MOC08,SW04} (we are grateful to N.Levenberg
for pointing this out). Indeed, the
space of spherical polynomials of total degree at most $k$ is by
definition the image of the restriction to $S^{n}$ of the space of all
polynomials on $\R^{n+1}$ of degree at most $k$. It thus coincides with (the real points of) $H^{0}(X,kL)$
with $X$ being the smooth quadric hypersurface 
$$
\{X_{1}^{2}+..+X_{n}^{2}=X_{0}^{2}\}\subset\PP^{n+1}
$$
endowed with the ample line bundle $L:=\cO(1)|_{X}$. Here we take
$K:=S^{n}=X(\R)$, and the section cutting out the hyperplane at infinity
again identifies weights on $L$ with certain functions on the affine
piece of $X$.

In view of the above dictionary, one is naturally led to introduce
the \emph{equilibrium weight} of $(K,\phi)$ as 

\begin{equation}\label{equ:equilib}
\phi_{K}:=\sup\left\{ \psi\,\text{ psh weight on }L,\,\psi\leq\phi\,\,\text{ on }K\right\},\end{equation}
 whose upper semi-continuous regularization $\phi_{K}^{*}$ is a psh
weight on $L$ since $K$ is non-pluripolar (cf.~Section \ref{sec:polar}).

The \emph{equilibrium measure} of $(K,\phi)$ is then defined as the
Monge-Amp{\`e}re measure of $\phi_{K}^{*}$ normalized to unit mass: 
$$
\eq(K,\phi):=M^{-1}\MA(\phi_{K}^{*}).
$$
 This measure is concentrated on $K$, and we have $\phi=\phi_{K}^{*}$
a.e. with respect to it.

This approach is least technical when $L$ is \emph{ample}, but the
natural setting appears to be the more general case of a \emph{big}
line bundle, which is the one considered in the present paper, following
our preceding work \cite{BB08a}. As was shown there, the Monge-Amp{\`e}re
measure $\MA(\psi)$ of a psh weight $\psi$ with minimal singularities,
defined as the Beford-Taylor top-power $(dd^{c}\psi)^{n}$ of the
curvature $dd^{c}\psi$ on its bounded locus, is well-behaved. Its
total mass $M$ is in particular an invariant of the big line bundle
$L$, and in fact coincides with the \emph{volume} $\vol(L)$, characterized
by 
$$
N_k:=\dim H^0(kL)=\vol(L)\frac{k^n}{n!}+o(k^n).
$$

The main goal of the present paper is to give a general criterion
involving spaces of global sections that ensures convergence of certain
sequences of probability measures on $K$ towards the equilibrium
measure $\eq(K,\phi)$.

\subsection{Fekete configurations}

Let $(K,\phi)$ be a weighted compact subset as above. A \emph{Fekete
configuration} is a finite subset of points maximizing the determinant
in the interpolation problem. More precisely, let $N:=\dim H^0(L)$ and 
$$
P=(x_1,...,x_N)\in K^N
$$
be a configuration of points in the given compact subset $K$. Then
$P$ is said to be a Fekete configuration for $(K,\phi)$ if it maximizes
the determinant of the evaluation operator

\begin{equation}\label{equ:ev}
\ev_P:H^0(L)\rightarrow\oplus_{j=1}^N L_{x_j}
\end{equation}
with respect to a given basis $s_1,...,s_N$ of $H^0(L)$,
i.e. the Vandermonde-type determinant 
$$
\left|\det(s_{i}(x_{j}))\right|e^{-\left(\phi(x_{1})+...+\phi(x_{n})\right)}.
$$
This condition is independent of the choice of the basis $(s_j)$.

If $P=(x_1,...,x_N)\in X^N$ is a configuration, then we let
$$
\delta_P:=\frac{1}{N}\sum_{j=1}^N\delta_{x_j}
$$
 be the averaging measure along $P$. Our first main result is an equidistribution
result for Fekete configurations.

\begin{thmA} For each $k$ let $P_k\in K^{N_k}$ be a Fekete configuration for $(K,k\phi)$. Then the sequence $P_k$ equidistributes towards the equilibrium measure as $k\to\infty$, that is
$$
\lim_{k\to\infty}\delta_{P_k}=\eq(K,\phi)
$$
in the weak topology of measures.
\end{thmA}

Theorem A first appeared in the first two named authors' preprint \cite{BB08b}.
It will be obtained here as a consequence of a more general convergence
result (Theorem C below).

In $\C$ this result is well-known (cf. \cite{ST97} for a modern
reference and \cite{Dei99} for the relation to Hermitian random matrices).
In $\C^{n}$ this result has been conjectured for quite some time,
probably going back to the pioneering work of Leja in the late 50's.
See for instance Levenberg's survey on approximation theory in $\C^{n}$ \cite{Lev06},
p.29 and the appendix by Bloom in \cite{ST97}.

As explained above, the spherical polynomials situation corresponds
to the round sphere $S^{n}$ embedded in its complexification, the
complex quadric hypersurface in $\PP^{n+1}$. This special case of
Theorem A thus yields:

\begin{corA} Let $K\subset S^n$ be a compact subset of the round \(n\)-sphere, and for each \(k\) let \(P_k\in K^{N_k}\) be Fekete configuration of degree \(k\) for \(K\) (also called \emph{extremal fundamental system} in this setting). Then \(\delta_{P_k}\) converges to the equilibrium measure \(\eq(K)\) of \(K\). 
\end{corA}

This is a generalization of the recent result of Morza and Ortega-Cerd{\`a} \cite{MOC08}
on equidistribution of Fekete points on the sphere. Their result corresponds
to the case $K=S^{n}$ whose equilibrium measure $\eq(S^{n})$ coincides
with the rotationally invariant probability measure on $S^{n}$ for
symmetry reasons.

\subsection{Bernstein-Markov measures}

Let as before $(K,\phi)$ be a weighted compact subset, and let $\mu$
be a probability measure on $K$. The distortion between the natural
$L^{2}$ and $L^{\infty}$ norms on $H^{0}(L)$ introduced above is
locally acounted for by the \emph{distortion function} $\rho(\mu,\phi)$,
whose value at $x\in E$ is defined by \begin{equation}
\rho(\mu,\phi)(x)=\sup_{\Vert s\Vert_{L^{2}(\mu,\phi)}=1}|s(x)|^{2}e^{-2\phi(x)},\label{equ:distortion}\end{equation}
 the squared norm of the evaluation operator at $x$.

The function $\rho(\mu,\phi)$ is known as the \emph{Christoffel-Darboux
function} in the orthogonal polynomials literature and may also be
represented as \begin{equation}
\rho(\mu,\phi)(x)=\sum_{i=1}^{N}|s_{i}(x)|^{2}e^{-2\phi(x)}\label{eq:rho in base}\end{equation}
 in terms of any given orthonormal base $(s_{i})$ for $H^{0}(L)$
wrt the $L^{2}-$norm induced by $(\mu,\phi).$ In this latter form,
it sometimes also appears under the name \emph{density of states function}.
Integrating (\ref{eq:rho in base}) shows that the corresponding \emph{probability}
measure \begin{equation}
\beta(\mu,\phi):=N^{-1}\rho(\mu,\phi)\mu,\label{equ:bergmes}\end{equation}
 which will be referred to as the \emph{Bergman measure}, can indeed
be interpreted as a dimensional density for $H^{0}(L)$.

When $\mu$ is a smooth positive volume form on $X$ and $\phi$ is
smooth and strictly psh, the celebrated Bouche-Catlin-Tian-Zelditch
theorem (\cite{Bou90,Cat99,Tia90,Zel98}) asserts that $\beta(\mu,k\phi)$
admits a full asymptotic expansion in the space of smooth volume forms
as $k\rightarrow\infty$, with $M^{-1}(dd^{c}\phi)^{n}$ as the dominant
term.

As was shown by the first named author (in \cite{Ber07a} for the
$\PP^{n}$ case and in \cite{Ber07b} for the general case), part
of this result still holds when $\mu$ is a smooth positive volume
form and $\phi$ is smooth but without any \emph{a priori} curvature
sign. More specifically, the norm distortion still satisfies \begin{equation}
\sup_{X}\rho(\mu,k\phi)=O(k^{n})\label{equ:berg1}\end{equation}
 and the Bergman measures still converge towards the equilibrium measure:
\begin{equation}
\lim_{k\rightarrow\infty}\beta(\mu,k\phi)=\eq(X,\phi)\label{equ:berg2}\end{equation}
 now in the weak topology of measures.

Both of these results fail when $K,\mu$ and $\phi$ are more general.
However \emph{sub-exponential} growth of the distortion between $L^{2}(\mu,k\phi)$
and $L^{\infty}(K,k\phi)$ norms, that is \begin{equation}
\sup_{K}\rho(\mu,k\phi)=O(e^{\e k})\,\,\mathrm{for}\,\,\mathrm{all}\,\,\e>0,\label{equ:berg3}\end{equation}
 appears to be a much more robust condition. Following a standard
terminology (cf. \cite{NZ83} and \cite{Lev06} p.120), the measure
$\mu$ will be said to be \emph{Bernstein-Markov} for $(K,\phi)$
when (\ref{equ:berg3}) holds.

When $K=X$, any measure with continuous positive density is Bernstein-Markov
for $(X,\phi)$ by the mean-value inequality. Generalizing classical
results of Nguyen-Zeriahi and Siciak (\cite{NZ83,Sic88}), we give
more generally in Section \ref{sec:BM} a characterization of Bernstein-Markov
measures showing that Bernstein-Markov measures for $(K,\phi)$ abound
when the latter is \emph{regular} in the sense of pluripotential theory,
i.e.~when $\phi_{K}$ is usc. For instance any smoothly bounded
domain $K$ in $X$ is regular, and we show that the equilibrium measure
of $(K,\phi)$ as well as any measure with support equal to $K$ is
Bernstein-Markov. 

Our second main result asserts that convergence of Bergman measures
to equilibrium as in (\ref{equ:berg2}) holds for arbitrary Bernstein-Markov
measures. 

\begin{thmB} Let \(\mu\) be a Bernstein-Markov measure for \((K,\phi)\). Then 
\[\lim_{k\to\infty}\beta(\mu,k\phi)=\eq(K,\phi)\]
in the weak topology of measures.
\end{thmB}

In the classical one-variable setting, this theorem was obtained,
using completely different methods, by Bloom and Levenberg \cite{BL07a}, who also conjectured Theorem B in \cite{BL07b}.
A slightly less general version of Theorem B (dealing only with \emph{stably}
Bernstein-Markov measures) was first obtained in the first and third
named author's preprint \cite{BWN08}. Theorem B will here be obtained
as a special case of Theorem C below.

\subsection{Donaldson's $\cL$-functionals and general convergence criterion}

\label{sub:donald} We now state our third main result, which is a
general criterion ensuring convergence of Bergman measures to equilibrium
in terms of $\cL$-functionals, first introduced by Donaldson \cite{Don05a,Don05b}.
This final result actually implies Theorem A and B above, as well
as a convergence result for so-called \emph{optimal measures} first
obtained in \cite{BBLW08} by reducing the result to \cite{BB08b}.

The $L^{2}$ and $L^{\infty}$ norms on $H^{0}(kL)$ introduced above
are described geometrically by their unit balls, which will be denoted
respectively by \[
\cB^{\infty}(K,k\phi)\subset\cB^{2}(\mu,k\phi)\subset H^{0}(kL).\]
 We fix a reference weighted compact subset $(K_{0},\phi_{0})$, which
should be taken to be the compact torus endowed with the standard
flat weight in the classical $\C^{n}$ case. We can then normalize
the Haar measure $\vol$ on $H^{0}(kL)$ by \[
\vol\cB^{\infty}(K_{0},k\phi_{0})=1,\]
 and we introduce the following slight variants of Donaldson's $\cL$-functional \cite{Don05a}:
\[
\cL_{k}(\mu,\phi):=\frac{1}{2kN_{k}}\log\vol\cB^{2}(\mu,k\phi)\]
 and \[
\cL_{k}(K,\phi):=\frac{1}{2kN_{k}}\log\vol\cB^{\infty}(K,k\phi).\]
 Theorem A of \cite{BB08a} then reads \begin{equation}
\lim_{k\rightarrow\infty}\cL_{k}(K,\phi)=\eneq(K,\phi).\label{equ:BB}\end{equation}
 Here \[
\eneq(K,\phi):=M^{-1}\cE(\phi_{K}^{*})\]
 denotes the \emph{energy at equilibrium} of $(K,\phi)$ (with respect
to $(K_{0},\phi_{0})$). $\cE(\psi)$ stands for the \emph{Aubin-Mabuchi
energy} of a psh weight $\psi$ with minimal singularities, characterized
as the primitive of the Monge-Amp{\`e}re operator: \[
\frac{d}{dt}_{t=0_{+}}\cE(t\psi_{1}+(1-t)\psi_{2})=\int_{X}(\psi_{1}-\psi_{2})\MA(\psi_{2})\]
 normalized by \[
\cE(\phi_{0,E_{0}}^{*})=0.\]

Since $\cL_{k}(\mu,\phi)\geq\cL_{k}(K,\phi)$ for any probability
measure $\mu$ on $K$, (\ref{equ:BB}) shows in particular that the
energy at equilibrium $\eneq(K,\phi)$ is an \emph{a priori} asymptotic
lower bound for $\cL_{k}(\cdot,\phi)$. Our final result describes
what happens for asymptotically minimizing sequences:

\begin{thmC} Let \(\mu_k\) be a sequence of probability measures on \(K\) such that
\[\lim_{k\to\infty}\cL_k(\mu_k,\phi)=\eneq(K,\phi).\]
Then the associated Bergman measures satisfy
\[\lim_{k\to\infty}\beta(\mu_k,k\phi)=\eq(K,\phi)\]
in the weak topology of measures.
\end{thmC}

The condition bearing on the sequence $(\mu_{k})$ in Theorem C is
independent of the choice of the reference weighted compact subset
$(K_{0},\phi_{0})$. In fact (\ref{equ:BB}) shows that it can equivalently
be written as the condition \[
\log\frac{\vol\cB^{2}(\mu_{k},k\phi)}{\vol\cB^{\infty}(K,k\phi)}=o(kN_{k}),\]
 which can be understood as a \emph{weak Bernstein-Markov condition}
on the sequence $(\mu_{k})$, relative to $(K,\phi)$, cf.~Lemma \ref{lem:vol}
below.

As will be clear from the proof of Theorem A the case when $\mu_{k}=\delta_{P_{k}}$
is equivalent to equidistribution of sequences $P_{k}$ of configurations
that are \emph{asymptotically Fekete} for $(K,\phi)$ in the sense
that \begin{equation}
\liminf_{k\rightarrow\infty}\frac{1}{kN_{k}}\log\left|(\det S_{k})(P_{k})\right|_{k\phi}\geq0\label{eq:asymptotically fekete}\end{equation}
 where $S_{k}$ is an orthonormal basis for $H^{0}(kL)$ wrt $(\mu_{0},k\phi_{0}).$
In order to get Theorem A we then use the simple fact that \begin{equation}
\beta(\mu,\phi)=\mu\label{eq:balanced in intro}\end{equation}
 for measures $\mu$ of the form $\delta_{P}$.

The proof of Theorem C is closely related to the generalization of
Yuan's equidistribution theorem for generic sequences of $\overline{Q}$-points \cite{Yua06}
obtained in \cite{BB08a}.

\subsection{Applications to interpolation}

Next, we will consider an application of Theorem C to a general \emph{interpolation
problem} for sections of $kL.$ The problem may be formulated as follows:
given a weigted set $(K,\phi)$ what is the distribution of $N_{k}$
(nearly) optimal \emph{interpolation nodes} on $K$ for elements in
$H^{0}(X,kL)?$ Of course, for any generic configuration $P_{k}$
the evaluation operator $\ev_{P_{k}}$ in (\ref{equ:ev}) is invertible
and interpolation is thus possible. But the problem is to find the
distribution of \emph{optimal} interpolation nodes, in the sense that
$P_{k}$ minimizes a suitable operator norm of the interpolation operator
$(\ev_{P_{k}})^{-1}$ over all configurations on $K$. More precisely,
given a measure $\mu$ supported on $K$ and numbers
$p,q$ such that $1\leq p,q\leq\infty,$ the \emph{$L^{p}(\mu,k\phi)$-$L^{q}(\de_{P_k},k\phi)$
distortion}: \begin{equation}
\sup_{s\in H^{0}(kL)}\frac{\left\Vert s\right\Vert _{L^{p}(\mu,k\phi)}}{\left\Vert s\right\Vert _{L^{q}(\delta_{P_{k}},k\phi)}},\label{eq:dist}\end{equation}
 defines a function on the space $K^{N_{k}}$ of all configurations
$P_{k}$ on $K.$ A configuration $P_{k}$ in $K^{N_{k}}$ will be
said to be \emph{optimal} wrt $(\mu;p,q)$ if it minimizes the corresponding
distortion over all configurations on $K.$ It should be pointed out
that it is in practice virtually impossible to find such optimal configurations
numerically. But the next corollary give necessary conditions for any
sequence of configurations to have \emph{sub-exponential} distortion
and in particular to be optimal. 

\begin{corC}
Let $\mu$ be a Bernstein-Markov measure for the weighted regular
set $(K,\phi)$ and let $(P_{k})$ be a sequence of configurations
in $K^{N_{k}}$ such that the distortion $(\ref{eq:dist})$ has subexponential
growth in $k$ for given numbers $p,q\in[1,\infty].$ Then $(P_{k})$
is asymptotically equilibrium distributed i.e. \[
\lim_{k\rightarrow\infty}\delta_{P_{k}}=\eq(K,\phi)\]
in the weak topology of measures. Moreover, any sequence of optimal
configurations wrt $(\mu;p,q)$\emph{ }has subexponential distortion
and is hence\emph{ }asymptotically equilibrium distributed.
\end{corC}
Note that we have assumed that $(K,\phi)$ is regular to make sure
that when $p=\infty$ the $L^{\infty}$ norm wrt $\mu$ coincides
with the sup-norm on $K$ (cf.~Proposition~\ref{prop:deter} below).  

It should be pointed out that the content of the corollary above is
well-known in the classical one-variable setting in $\C$ corresponding
to the case $(X,L)=(\PP^{1},\mathcal{O}(1)).$ Indeed in the latter
setting the case where $K$ is a compact subset of the real line $\R\subset\C$
was treated in \cite{GMS02}. The general one-dimensional case was
then obtained in \cite{BL03a}.

When $(p,q)=(\infty,2)$ the distortion is given by 
$$\sup_K\rho(\mu_{0},k\phi)^{1/2}.$$
Minimizers of the distortion among \emph{all}
probability measures were called \emph{optimal measures} (for $(K,k\phi)$)
in \cite{BBLW08}, where the convergence result for optimal measures
was obtained (by reducing the problem to the convergence of
Fekete points \cite{BB08b}). It turns out that optimal measures satisfy
(\ref{eq:balanced in intro}) and yield probability measures on $K$ that minimize the functional \emph{$\mathcal{L}(\cdot,\phi)$}
- see \cite{KW60} and Proposition \ref{prop:optimal} in our setting.
Such measures appear naturally in the context of \emph{optimal experimental
designs} (see \cite{BBLW08} and references therein).
\begin{rem}
For a numerical study in the setting of Corollary A and with $\mu_{0}$
the invariant measure on $S^{2}$ see \cite{SW04}, where the cases
$(p,q)=(\infty,\infty)$ and $(p,q)=(2,2)$ are considered. It should
also be pointed out that in the classical litterature on orthogonal
polynomials optimal configurations are usually called \emph{Lesbegue points} in the case $(p,q)=(\infty,\infty)$and \emph{Fejer points}
in the case $(p,q)=(\infty,2)$. 
\end{rem}

\subsubsection*{Recursively extremal configurations}

Finally, we will consider a recursive way of constructing configurations
with certain extremal properties. Even if the precise construction
seems to be new, it should be emphasized that it is inspired by the
elegant algorithmic construction of determinantal random point processes
in \cite{hkpv}.

Fix a weighted measure $(\mu,\phi)$ where $\mu$ is as before a probability
measure on $K$. A configuration $P=(x_{1},...,x_{N})$ will be said
to be \emph{recursively extremal} for $(\mu,\phi)$ if it arises in
the following way. Denote by $\mathcal{H}_{N}$ be the corresponding
Hilbert space $H^{0}(X,L)$ of dimension $N$. Take a pair $(x_{N},s_{N})$
maximizing the point-wise norm $\left|s(x)\right|_{\phi}^{2}$ over
all points $x$ in the set $K$ and sections $s$ in the unit-sphere
in $\mathcal{H}_{N}.$ Next, replace $\mathcal{H}_{N}$ by the Hilbert
space $\mathcal{H}_{N-1}$ of dimension $N-1$ obtained as the orthogonal complement
of $s_{N}$ in $\mathcal{H}_{N}$ and repeat the procedure to get
a new pair $(x_{N-1},s_{N-1})$ where now $s_{N-1}\in\mathcal{H}_{N-1}.$
Continuing in this way gives a configuration $P:=(x_{1},...,x_{N})$
after $N$ steps.

Note that $x_{N}$ may be equivalently obtained as a point maximizing
the Bergman distortion function $\rho(x)$ of $\mathcal{H}_{N}$ and
so on. Hence, the main advantage of recursively extremal configurations
over Fekete configurations, is that they are obtained by maximizing
functions defined on $X$ and not on the space $X^{N}$ of increasing
dimension. This advantage should make them useful in numerical interpolation
problems. We show that a sequence of recursively extremal configurations
$P_{k}$ is, in fact, \emph{asymptotically Fekete} (i.e. (\ref{eq:asymptotically fekete})
holds). As a direct consequence $P_{k}$ is equilibrium distributed:
\begin{corD}
Let $\mu$ be a Bernstein-Markov measure for the weighted set $(K,\phi)$
and $P_{k}$ a sequence of configurations which are \emph{recursively
extremal} for $(\mu,k\phi).$ Then \[
\lim_{k\rightarrow\infty}\delta_{P_{k}}=\eq(K,\phi)\]
 in the weak topology of measures. 
\end{corD}
\begin{ackn}
It is a pleasure to thank Bo Berndtsson, Jean-Pierre Demailly and Norm Levenberg for stimulating discussions related to the topic of this article. 
\end{ackn}

\section{Regular sets and Bernstein-Markov measures}

Recall that $L$ denotes a given big line bundle over a complex compact
manifold $X$. The existence of a such a big line bundle on $X$ is
equivalent to $X$ being \emph{Moishezon}, i.e.~bimeromorphic to
a projective manifold, and $X$ is then projective iff it is K{\"a}hler.

\subsection{Pluripolar subsets and regularity}

\label{sec:polar} The goal of this section is to recall some preliminary
results from \cite{BB08a} and to quickly explain how to adapt to our
big line bundle setting further results on equilibrium weights that
are standard in the classical situation. We refer to Klimek's book \cite{Kli91}
and Demailly's survey \cite{Dem} for details.

First recall that a set $A$ in $X$ said to be \emph{(locally) pluripolar}
if it is locally contained in the polar set of a local psh function.
For a big line bundle $L$ this is equivalent to the following global
notion of pluripolarity (as shown by Josefson in the classical setting):
\begin{prop}
\label{prop:josef} If $A\subset X$ is (locally) pluripolar, then
there exists a psh weight $\phi$ on $L$ such that $A\subset\{\phi=-\infty\}$. \end{prop}
\begin{proof}
Since $L$ is big, we can find an effective divisor $E$ with $\Q$-coefficients
such that $L-E$ is ample. By Guedj-Zeriahi's extension of Josefson's
result to the K{\"a}hler situation \cite{GZ05}, there exists a closed
positive $(1,1)$-current $T$ cohomologous to $L-E$ whose polar
set contains $A$. We can thus find a psh weight $\phi$ on $L$ such
that $dd^{c}\phi=T+[E]$, and the polar set of $\phi$ contains $A$
as desired. 
\end{proof}
By definition, a \emph{weighted compact subset} $(K,\phi)$ consists
of a non-pluripolar compact set $K\subset X$ together with a continuous
Hermitian metric $e^{-\phi}$ on $L|_{K}$. By the Tietze-Urysohn
extension theorem, $\phi$ extends to a continuous weight on $L$
over all of $X$. Now if $E$ is an \emph{arbitrary} subset of $X$
and $\phi$ is a continuous weight on $L$ (over all of $X$) we define
the associated extremal function by \[
\phi_{E}:=\sup\left\{ \psi\,\text{ psh weight on }L,\,\psi\leq\phi\,\,\text{ on }E\right\} .\]

It is shown (see \cite{GZ05}) that its usc regularization $\phi_{E}^{*}$
is a psh weight if $E$ is non-pluripolar, whereas $\phi_{E}^{*}\equiv+\infty$
if $E$ is pluripolar.

It is enough to consider psh weights induced by sections in the definition
of $\phi_{E}^{*}$: 
\begin{prop}
\label{prop:equisec} Let $E$ be a non-pluripolar subset of $X$
and let $\phi$ be a continuous weight. Then we have \begin{equation}
\phi_{E}^{*}=\mathrm{sup}^{*}\{\frac{1}{k}\log|s|,\, k\geq1,\, s\in H^{0}(kL),\,\sup_{E}|s|_{k\phi}\leq1\}.\label{equ:equisec}\end{equation}

\end{prop}
In the classical case, the result is known to hold even without taking
usc regularization on both sides (cf.~for instance Theorem 5.6 in
the survey \cite{Dem}). As we shall see in the proof this is more
generally the case when $L$ is ample, but we do not know whether it
remains true for an arbitrary big line bundle.
\begin{proof}
Let $\psi$ denote the right-hand side of (\ref{equ:equisec}). It
is clear by definition that $\phi_{E}^{*}\geq\psi$. To get the converse
inequality we shall apply the Ohsawa-Takegoshi extension theorem.
Let thus $\tau$ be a psh weight on $L$ such that $\tau\leq\phi$
on $E$. In order to show that $\tau\leq\psi$, we may furthermore
assume that $\tau$ is \emph{strictly psh} in the sense that $dd^{c}\tau$
dominates a smooth strictly positive $(1,1)$-form. Indeed there exists
a strictly psh weight $\psi_{+}$ since $L$ is big, and $(1-\e)\tau+\e\psi_{+}$
is then strictly psh for every $\e>0$. If we know that \[
(1-\e)\tau+\e\psi_{+}\leq\psi\]
 for each $\e>0$ then $\tau\leq\psi$ holds outside the polar set
of $\psi_{+}$, which has Lebesgue measure $0$, and we infer $\tau\leq\psi$
everywhere on $X$ as desired.

Note that it is at this point that we need to take usc regularizations
in (\ref{equ:equisec}) since $\psi_{+}$ will have poles in general
when $L$ is not ample. 

Since $\tau$ is strictly psh, the Ohsawa-Takegoshi-Manivel $L^{2}$-extension
theorem (see for instance \cite{Bern05} for a particularly nice approach) yields a constant
$C>0$ such that for every point $x_{0}\in X$ outside the polar set
of $\tau$ we can find a sequence
$s_{k}\in H^{0}(kL)$ with \begin{equation}
\int_{X}|s_{k}|^{2}e^{-2k\tau}d\lambda\leq C|s_{k}(x_{0})|^{2}e^{-2k\tau(x_{0})}\label{equ:OT}\end{equation}
 for all $k\gg1$ (more precisely for all $k$ such that the strictly
positive current $kdd^{c}\tau$ absorbs the curvature of some given
smooth metric on the canonical bundle). We have denoted by $\lambda$
a Lebesgue measure on $X$. On the other hand given $\e>0$ we have
$\tau\leq\phi+\e$ on a neighbourhood of $E$ by upper semi-continuity
of $\tau-\phi$. Since $\lambda$ is Bernstein-Markov for $(X,\phi)$
(an easy consequence of the mean value inequality, cf. \cite{BB08a})
there exists another constant $C_{1}>0$ such that \[
\sup_{E}|s_{k}|_{k\phi}^{2}\leq\sup_{X}|s_{k}|_{k\phi}^{2}\]
 \[
\leq C_{1}e^{k\e}\int_{X}|s_{k}|^{2}e^{-2k\phi}d\lambda\leq C_{1}e^{3k\e}\int_{X}|s_{k}|^{2}e^{-2k\tau}d\lambda\]
 which is in turn \[
\leq C_{2}e^{3k\e}|s_{k}(x_{0})|^{2}e^{-2k\tau(x_{0})}\]
 for all $k\gg1$ by (\ref{equ:OT}). But this means that \[
\tau(x_{0})\leq\frac{1}{k}\log|\tilde{s}_{k}(x_{0})|+\frac{3}{2}\e+\frac{C_{3}}{k}\]
 where \[
\tilde{s}_{k}:=\frac{s_{k}}{\sup_{E}|s_{k}|_{k\phi}}\]
 has $\sup_{E}|\tilde{s}_{k}|_{k\phi}=1$ hence is a candidate in
the right-hand side of (\ref{equ:equisec}). It therefore follows that
\[
\tau(x_{0})\leq\psi(x_{0})+2\e\]
 for every $\e>0$ and every $x_{0}$ outside the polar
locus of $\tau$, which has measure
$0$, and the result follows. 
\end{proof}
Using Proposition \ref{prop:josef} one proves the following two
useful facts exactly as in the classical setting (cf.~for instance \cite{Kli91},
p.194): 
\begin{prop}
\label{prop:polar} Let $\phi$ be a continuous weight and let $E,A\subset X$
be two subsets with $A$ pluripolar, then we have $\phi_{E\cup A}^{*}=\phi_{E}^{*}$. \end{prop}
\begin{cor}
\label{cor:decr} If $E$ is the increasing union of subsets $E_{j}$,
then $\phi_{E_{j}}^{*}$ decreases pointwise to $\phi_{E}^{*}$ as
$j\rightarrow\infty$. 
\end{cor}
Adapting to our setting a classical notion we introduce \begin{defi} If \(E\) is a non-pluripolar subset of \(X\) and \(\phi\) is a continuous weight, we say that \((E,\phi)\) is \emph{regular} (or that \(E\) is regular with respect to \(\phi\)) iff \(\phi_E\) is upper semi-continuous.
\end{defi} As opposed to the classical case (cf. \cite{Dem} Theorem 15.6),
we are unable to prove that $\phi_{E}$ is \emph{a priori} lower semi-continuous
when $L$ is not ample, hence our definition (note that $\phi_{E}^{*}$
has $-\infty$-poles in the big case, see \cite{BB08a} Remark 1.14
for a short discussion on this issue).

Note that $\phi_{E}$ is usc iff $\phi_{E}^{*}$ satisfies $\phi_{E}^{*}\leq\phi$
on $E$, that is iff the set of psh weights $\psi$ such that $\psi\leq\phi$
on $E$ admits a largest element.

Since $\phi$ is in particular usc, we see that $X$ or in fact any
open subset of $X$ is regular with respect to $\phi$.

Irregularity of a subset $E$ with respect to $\phi$ is always accounted
for by a pluripolar set. Indeed the set of points where $\phi_{E}<\phi_{E}^{*}$
is negligible, hence pluripolar by \cite{BT82} Theorem 7.1. Conversely,
a typical example of an irregular set is obtained by adding to a given
subset a pluripolar one, in view of Proposition \ref{prop:polar}.
For instance in the classical situation, adding to the closed
unit disk of $\C$ an outside point yields an irregular set.

In order to get examples of regular sets in our setting, we shall
say that a compact subset $K\subset X$ is \emph{locally regular}
if for every open set $U$ and every local non-decreasing uniformly bounded sequence $u_{j}$
of psh functions on $U$ such that $u_{j}\leq0$ on $K\cap U$ the usc upper envelope
also satisfies \[
(\sup_{j}u_{j})^{*}\leq0\text{ on }K\cap U.\]
 This notion is independent of the line bundle $L$, and means in
fact that every point of $K$ sits in a small ball $B$ such that
the \emph{relative extremal function} \[
u_{K\cap B,B}:=\sup\{v\leq0\text{ psh on }B,\, v\leq-1\text{ on }K\cap B\}\]
 of $K\cap B$ in $B$ (cf. \cite{Kli91} P.158) satisfies \[
u_{K\cap B,B}^{*}\equiv-1\text{ on }K\cap B.\]
 The following criterion is easily checked. 
\begin{lem}
\label{lem:locreg} Let $K\subset X$ be a non-pluripolar compact
subset and assume that $K$ is locally regular. Then $K$ is regular
with respect to every continuous weight $\phi$ on $L$. 
\end{lem}
The converse implication already fails already in the classical situation,
cf. \cite{Sad81}.

We have the so-called accessibility criterion for regularity.
\begin{prop}
If $K\subset X$ is a compact subset of $X$ and if for each point
$x\in\partial K$ in the topological boundary there exists a real
analytic arc $\gamma:[0,1]\rightarrow X$ such that $\gamma(0)=z$
whereas $\gamma(]0,1])$ is contained in the topological interior
$K^{0}$, then $K$ is locally regular. 
\end{prop}
This is well-known and follows from the fact that any subharmonic
function $u$ defined around $[0,1]\subset\C$ satisfies \[
u(0)=\limsup_{z\rightarrow0,\, z\in]0,1]}u(z).\]

\begin{cor}
\label{cor:smooth} Let $\Omega$ and $M$ be a smoothly bounded
domain and a real analytic $n$-dimensional
totally real submanifold respectively. Then $\overline{\Omega}$
and $M$ are locally regular. \end{cor}
\begin{proof}
The first assertion follows from the accessibility criterion just
as in \cite{Kli91} Cor 5.3.13 and the second from the fact that
$\R^{n}$ is locally regular in $\C^{n}$.
\end{proof}
It seems to be unknown whether the real analyticity assumption on
$M$ can be relaxed to $C^{\infty}$ regularity.

\subsection{Bernstein-Markov and determining measures}

\label{sec:BM} Recall from the introduction that given a weighted
compact subset $(K,\phi)$ we say that a probability measure $\mu$
on $K$ is \emph{Bernstein-Markov} for $(K,\phi)$ iff the distortion
between the $L^{\infty}(K,k\phi)$ and $L^{2}(\mu,k\phi)$ norms on
$H^{0}(kL)$ has sub-exponential growth as $k\rightarrow\infty$,
that is:

For each $\e>0$ there exists $C>0$ such that \begin{equation}
\sup_{K}|s|^{2}e^{-2k\phi}\leq Ce^{\e k}\int|s|^{2}e^{-2k\phi}d\mu.\label{equ:BMsections}\end{equation}
 for each $k$ and each section $s\in H^{0}(kL)$.

We are going to obtain a characterization of the following stronger
property. \begin{defi}\label{defi:BMpsh} Let \((K,\phi)\) be a weighted compact subset and let \(\mu\) be a non-pluripolar probability measure on \(K\). Then \(\mu\) will be said to be \emph{Bernstein-Markov with respect to psh weights} for \((K,\phi)\) iff for each \(\e>0\) there exists \(C>0\) such that 
\begin{equation}\label{equ:BMpsh}
\sup_K e^{p(\psi-\phi)}\leq Ce^{\e p}\int e^{p(\psi-\phi)}d\mu.
\end{equation}
for all \(p\ge 1\) and all psh weights \(\psi\) on \(L\). 
\end{defi} One virtue of Definition \ref{defi:BMpsh} is that it also makes
sense in the more general situation of $\theta$-psh functions with
respect to a smooth $(1,1)$-form $\theta$ as considered for example
in \cite{GZ05,BEGZ08}. It is immediate to see that $\mu$ is Bernstein-Markov
for $(K,\phi)$ (in the previous sense, i.e.~with respect to sections)
if it is Bernstein-Markov with respect to psh weights (apply the definition
to $\psi:=\frac{1}{k}\log|s|$ and $p:=2k$).

We will only consider \emph{non-pluripolar} measures $\mu$, that
is measures putting no mass on pluripolar subsets. Note that the equilibrium
measure $\eq(K,\phi)$ is non-pluripolar, since it is defined as the
non-pluripolar Monge-Amp{\`e}re measure of $\phi_{K}^{*}$ (cf. \cite{BB08a}).

Following essentially \cite{Sic88} we shall say that $\mu$ is \emph{determining}
for $(K,\phi)$ iff the following equivalent properties hold (compare \cite{Sic88}
Theorem A). 
\begin{prop}
\label{prop:deter} Let $(K,\phi)$ be a weighted compact subset and
let $\mu$ be a non-pluripolar probability measure on $K$. Then the
following properties are equivalent: 
\begin{enumerate}
\item[(i)] $(K,\phi)$ is regular and each Borel subset $E\subset K$ such
that $\mu(K-E)=0$ satisfies $\phi_{E}=\phi_{K}.$
\item[(ii)] For each psh weight $\psi$ we have \begin{equation}
\psi\leq\phi\,\mu\text{-a.e.}\Longrightarrow\psi\leq\phi\text{ on }K\label{equ:deter}\end{equation}

\item[(iii)] $(K,\phi)$ is regular and for each $k$ and each section $s\in H^{0}(kL)$
we have \[
\Vert s\Vert_{L^{\infty}(\mu,k\phi)}=\sup_{K}|s|_{k\phi}.\]

\end{enumerate}
\end{prop}
\begin{proof}
Assume that (i) holds and let $\psi$ be a psh weight such that $\psi\leq\phi$
$\mu$-a.e. Consider the Borel subset \[
E:=\{x\in K|\psi(x)\leq\phi(x)\}\subset K.\]
 We then have $\mu(K-E)=0$ by assumption, hence $\phi_{E}=\phi_{K}$.
On the other hand we have $\psi\leq\phi$ on $E$ by definition of
$E$, hence $\psi\leq\phi_{E}=\phi_{K}$ on $X$, and we infer $\psi\leq\phi$
on $K$. We have thus shown that (i)$\Rightarrow$(ii).

Assume that (ii) holds. The set $\{\phi_{K}<\phi_{K}^{*}\}$ is negligible
hence pluripolar by Bedford-Taylor's theorem, thus it has $\mu$-measure
$0$ since $\mu$ is non-pluripolar by assumption. We thus have $\phi_{K}^{*}=\phi_{K}\leq\phi$
$\mu$-a.e., and (ii) implies $\phi_{K}^{*}\leq\phi$ everywhere on
$K$, which means that $(K,\phi)$ is regular. On the other hand it
is straightforward to see that (ii) is equivalent to \[
\Vert e^{\psi-\phi}\Vert_{L^{\infty}(\mu)}=\sup_{K}e^{\psi-\phi}\]
 for all psh weights $\psi$, and (ii)$\Rightarrow$(iii) follows by
applying this to $\psi:=\frac{1}{k}\log|s|$.

Now assume that (iii) holds and let $E\subset K$ be a Borel set such
that $\mu(K-E)=0$. Since $\mu(K-E)=0$ each section $s\in H^{0}(kL)$
such that $\sup_{E}|s|_{k\phi}\leq1$ atisfies in particular $\Vert s\Vert_{L^{\infty}(\mu,k\phi)}\leq1$,
hence $\sup_{K}|s|e^{-k\phi}\leq1$ by (iii), i.e. \[
\frac{1}{k}\log|s|\leq\phi_{K}.\]
 By Proposition \ref{prop:equisec} we thus get \[
\phi_{E}\leq\phi_{E}^{*}\leq\phi_{K}^{*}=\phi_{K}\text{ on }K,\]
 hence $\phi_{E}=\phi_{K}$ since $E\subset K$ implies $\phi_{K}\leq\phi_{E}$. \end{proof}
\begin{cor}
\label{cor:support} Let $(K,\phi)$ be a regular weighted subset.
Then every non-pluripolar probability measure $\mu$ on $K$ such
that $\supp\mu=K$ is determining for $(K,\phi)$. \end{cor}
\begin{proof}
Since $K=\supp\mu$, the complement of every $\mu$-negligible subset
$A\subset K$ is dense in $K$, thus the $\mu$-essential supremum
of every $f\in C^{0}(K)$ coincides with its supremum on $K$. Since
$|s|_{k\phi}$ is continuous for every $s\in H^{0}(kL)$, it follows
that condition (iii) of Proposition \ref{prop:deter} is satisfied. \end{proof}
\begin{prop}
\label{prop:determining} Let $(K,\phi)$ be a weighted compact subset.
Then the equilibrium measure $\eq(K,\phi)$ is determining for $(K,\phi)$
iff $(K,\phi)$ is regular. \end{prop}
\begin{proof}
Suppose that $(K,\phi)$ is regular. The domination principle, itself
an easy consequence of the so-called comparison principle, states
that given two psh weights $\psi,\psi'$ on $L$ such that $\psi$
has minimal singularities we have ($\psi'\leq\psi$ a.e.~for $\MA(\psi)$)$\Rightarrow$($\psi'\leq\psi$
on $X$) (cf. \cite{BEGZ08} Corollary 2.5 for a proof in our context).
Applying this to $\psi:=\phi_{K}^{*}$ immediately yields the result
since we have $\phi_{K}^{*}\leq\phi$ on $K$ by the regularity assumption.
The converse follows from Proposition \ref{prop:deter}. 
\end{proof}
Note that Proposition \ref{prop:determining} is not implied by Corollary \ref{cor:support}
since the support of $\eq(K,\phi)$ coincides with the \emph{Silov
boundary} of $(K,\phi)$ when the latter is regular (cf. \cite{BT87}
Theorem 7.1 in the classical case).

We now introduce the following version of the Bernstein-Markov property
for psh weights.

We can now state our main result in this section, which generalizes
in particular \cite{Sic88}. 
\begin{thm}
\label{thm:BM} Let $(K,\phi)$ be a weighted compact subset and let
$\mu$ be a non-pluripolar probability measure on $K$. Then the following
properties are equivalent. 
\begin{enumerate}
\item[(i)] $(K,\phi)$ is regular and $\mu$ is Bernstein-Markov for $(K,\phi)$
\item[(ii)] $\mu$ is Bernstein-Markov with respect to psh weights for $(K,\phi)$. 
\item[(iii)] $\mu$ is determining for $(K,\phi)$. 
\end{enumerate}
\end{thm}
This theorem gives in particular a conceptually simpler proof of the
main results of \cite{Sic88,NZ83}. 
\begin{proof}
For $p>0$ we introduce the functionals \[
F_{p}(\psi):=\frac{1}{p}\log\int_{X}e^{p(\psi-\phi)}d\mu=\log\Vert e^{\psi-\phi}\Vert_{L^{p}(\mu)}\]
 and \[
F(\psi):=\sup_{K}(\psi-\phi)\]
 defined on the set $\cP(X,L)$ of all psh weights $\psi$ on $L$.
For each $\psi$ $pF_{p}(\psi)$ is a convex function of $p$ by convexity
of the exponential (H{\"o}lder's inequality), and we have $pF_{p}(\psi)\rightarrow0$
as $p\rightarrow0_{+}$ by dominated convergence since $p(\psi-\phi)\rightarrow0$
$\mu$-a.e. ($\mu$ puts no mass on the polar set $\{\psi=-\infty\}$).
As a consequence $F_{p}(\psi)$ is a non-decreasing function of $p$,
and it converges towards \[
\log\Vert e^{\psi-\phi}\Vert_{L^{\infty}(\mu)}\]
 as $p\rightarrow+\infty$ by a basic fact from integration theory.
We can therefore reformulate (ii) and (iii) as follows: 
\begin{enumerate}
\item[(ii')] $F_p-F$ is bounded on $\cP(X,L)$, uniformly for $p\ge 1$ and $F_p \to F$ uniformly on $\cP(X,L)$ as $p \to \infty$.
\item[(iii')] $F_p \to F$ pointwise on $\cP(X,L)$.
\end{enumerate}
(ii') is just a reformulation of (ii), and that (iii') is equivalent to (iii) follows from (ii) of Proposition \ref{prop:deter}. This clearly shows that (ii)$\Rightarrow$(iii) by Proposition \ref{prop:deter} and
in particular that $(K,\phi)$ is necessarily regular when (ii) holds,
by Proposition \ref{prop:deter}. We thus see that (ii)$\Rightarrow$(i),
and (i)$\Rightarrow$(iii) is obtained in a similar fashion using
again Proposition \ref{prop:deter}. All that remains to show is
thus (iii')$\Rightarrow$(ii').

By Hartogs' lemma $F$ is \emph{upper semicontinuous} on $\cP(X,L)$.
On the other hand Lemma \ref{lem:cont} below says that $F_{p}$
is \emph{continuous} on $\cP(X,L)$ for each $p>0$, so that $F-F_{p}$
is usc on $\cP(X,L)$. Now the main point is that $F-F_{p}$ is invariant
by translation (by a constant), thus descends to a usc function on
\[
\cP(X,L)/\R\simeq\cT(X,L),\]
 the space of all closed positive $(1,1)$-currents lying in the cohomology
class $c_{1}(L)$, which is \emph{compact} (in the weak topology of
currents). 

By monotonicity we have $0\leq F-F_{p}\leq F-F_{1}$ when $p\geq1$.
But $F-F_{1}$ is usc on a compact set hence is bounded from above,
and it follows that $F-F_{p}$ is always uniformly bounded on $\cP(X,L)$
for $p\geq1$. By the above discussion it thus follows that (iii)$\Rightarrow$(ii)
amounts to the fact that $F_{p}$ converges to $F$ uniformly as soon
as pointwise convergence holds, which is a consequence of Dini's lemma
since $F-F_{p}$ is usc and non-increasing on $\cT(X,L)$ as a function of $p$. 
\end{proof}

\begin{lem}
\label{lem:cont} The functional $F_{p}:\cP(X,L)\rightarrow\R$ is
continuous for each $p>0$. 
\end{lem}
The proof relies on more or less standard arguments. 

\begin{proof}
Let $\psi_{k}\rightarrow\psi$ be a (weakly) convergent sequence in
$\cP(X,L)$. Then $\sup_{X}(\psi_{k}-\phi)$ is uniformly bounded,
thus $u_{k}:=e^{p(\psi_{k}-\phi)}$ is a uniformly bounded sequence.
We may thus assume upon extracting a subsequence that $\int_{X}u_{k}d\mu\rightarrow l$
for some $l\in\R$ and we have to show that $l=\int_{X}ud\mu$ with $u:=e^{p(\psi-\phi)}$.
Since the functions $u_{k}$ stay in a weakly compact subset of the
Hilbert space $L^{2}(\mu)$, the closed convex subsets \[
C_{k}:=\overline{\text{Conv}\{u_{j},j\geq k\}}\subset L^{2}(\mu)\]
 are weakly compact in $L^{2}(\mu)$, and it follows that there exists
$v$ lying in the intersection of the decreasing sequence of compact
sets $C_{k}$. For each $k$ we may thus find a finite convex combination
\[
v_{k}=\sum_{j\in I_{k}}t_{j}^{(k)}u_{j}\]
 with $I_{k}\subset[k,+\infty[$ such that $v_{k}\rightarrow v$ strongly
in $L^{2}(\mu)$. Note that \[
\lim_{k\rightarrow\infty}\int_{X}v_{k}d\mu=l\]
 since $\int_{X}u_{k}d\mu\rightarrow l$, thus we get $\int_{X}vd\mu=l$.

On the other hand the convergence $\psi_{k}\rightarrow\psi$ in $\cP(X,L)$
implies that \[
\tau_{k}:=\log\left(\sum_{j\in I_{k}}t_{j}^{(k)}e^{p\psi_{j}}\right)\in\cP(X,pL)\]
 (the latter space is to be understood in the sense of quasi-psh functions
when $p$ is not an integer) converges to $p\psi$, and it follows
from Hartogs' lemma that \[
p\psi=(\limsup_{k\rightarrow\infty}\tau_{k})^{*}\]
 pointwise on $X$. But the set \[
\{(\limsup_{k\rightarrow\infty}\tau_{k})^{*}>\limsup_{k\rightarrow\infty}\tau_{k}\}\]
 is negligible, hence pluripolar by Bedford-Taylor, and we get \[
u:=e^{p(\psi-\phi)}=\limsup_{k\rightarrow\infty}v_{k}\,\mu-\text{a.e.}\]
 since $v_{k}=e^{\tau_{k}-p\phi}$. But since $v_{k}\rightarrow v$
in $L^{2}(\mu)$ there exists a subsequence such that $v_{k}\rightarrow v$
$\mu$-a.e., and we infer $u=v$ $\mu$-a.e., which finally shows
that \[
l=\int_{X}vd\mu=\int_{X}ud\mu\]
 as desired. \end{proof}
\begin{cor}
If $(K,\phi)$ is a regular weighted compact subset, then $\psi\mapsto\sup_{K}(\psi-\phi)$
is continuous on $\cP(X,L)$. 
\end{cor}

\begin{proof}
By Proposition \ref{prop:determining} and Theorem \ref{thm:BM} the equilibrium measure $\mu:=\mu_{eq}(K,\phi)$ is Bernstein-Markov for $(K,\phi)$ when $(K,\phi)$ is regular. By Lemma \ref{lem:cont} the functionals $\log ||e^{\psi-\phi}||_{L^p(\mu)}$ are continuous and from Theorem \ref{thm:BM} we get that they converge uniformly to $\sup_K(\psi-\phi),$ and the continuity thus follows.
\end{proof}
In the case when $X=\mathbb{P}^1$ and $L=\mathcal{O}(1)$ the result in the previous corollary was obtained by different methods in \cite{ZZ09} (Lemma 26). The fact that the equilibrium
measure of a regular weighted set $(K,\phi)$ is Bernstein-Markov generalizes \cite{NZ83}.

\section{Volumes of balls}

\subsection{Convexity properties}

Let $(K,\phi)$ be a weighted compact subset and let $\mu$ be a probability
measure on $K$. The $L^{2}$-seminorm \[
\Vert s\Vert_{L^{2}(\mu,\phi)}^{2}:=\int_{X}|s|^{2}e^{-2\phi}d\mu\]
can then be viewed as a Hermitian metric $L^{2}(\mu,\phi)$
on the complex vector space $H^{0}(L)$. If we are given a basis $S=(s_{1},...,s_{N})$
of $H^{0}(L)$, a Hermitian metric $H$ on $H^{0}(L)$ can be identified
with its Gram matrix \[
\left(\langle s_{i},s_{j}\rangle_{H}\right)_{i,j}\in\Herm^{+}(N)\]
 with $N=h^{0}(L)$ as before, and its determinant satisfies \[
\det H=\frac{\vol\Pi_{S}}{\vol\Pi_{S'}},\]
 where $S'$ is an $H$-orthonormal basis and $\Pi_{S}$ is the unit
box in the corresponding real vector space, generated by the elements of $S$ (and similarly for $\Pi_{S'}$).
Since $\pi^{N}/N!$ is equal to the volume of the unit ball in $\C^{N}$,
we infer \begin{equation}
\log\frac{\vol\cB^{2}(\mu,\phi)}{\vol\Pi_{S}}=-\log\det L^{2}(\mu,\phi)+\log\frac{\pi^{N}}{N!}\label{equ:L-func}\end{equation}
 where $\det$ is defined wrt $S$.

Now let $\det S$ be the image of $s_{1}\wedge...\wedge s_{N}$ under
the natural map \[
\bigwedge^{N}H^{0}(X,L)\longrightarrow H^{0}(X^{N},L^{\boxtimes N}),\]
 that is the global section on $X^{N}$ locally defined by \[
(\det S)(x_{1},...,x_{N}):=\det(s_{i}(x_{j})).\]

Expanding out the determinant as in Lemma 5.27 of \cite{Dei99}, one
easily shows:
\begin{lem}
\label{lem:det} The $L^{2}$-norm of $\det S$ with respect to the
weight and measure induced by $\phi$ and $\mu$ satisfies \[
\Vert\det S\Vert_{L^{2}(\mu,\phi)}^{2}=N!\det L^{2}(\mu,\phi).\]

\end{lem}
On the other hand, a straightforward computation yields
\begin{lem}
If $P\in X^{N}$ is a configuration of points, then \[
\Vert\det S\Vert_{L^{2}(\delta_{P},\phi)}^{2}=\frac{N!}{N^{N}}\left|\det S\right|_{\phi}^{2}(P).\]

\end{lem}
Combining these results, we record
\begin{prop}
\label{prop:formules} We have \begin{equation}
\log\frac{\vol\cB^{2}(\mu,\phi)}{\vol\Pi_{S}}=-\log\Vert\det S\Vert_{L^{2}(\mu,\phi)}^{2}+N\log\pi.\label{equ:formule1}\end{equation}
 If $\mu=\delta_{P}$, then \begin{equation}
\log\frac{\vol\cB^{2}(\delta_{P},\phi)}{\vol\Pi_{S}}=-\log|\det S|_{\phi}^{2}(P)+\log\frac{\pi^{N}}{N!}+N\log N.\label{equ:formule2}\end{equation}

\end{prop}
\noindent Note that the last formula reads \begin{equation}
\log\frac{\vol\cB^{2}(\delta_{P},\phi)}{\vol\cB^{2}(\nu,\psi)}=-\log|\det S|_{\phi}^{2}(P)+N\log N\label{equ:formule3}\end{equation}
 when $S$ is an orthonormal basis for $L^{2}(\nu,\psi)$.

The volume of balls satisfies the following convexity properties.
\begin{prop}
\label{prop:concave} Let $(K,\phi)$ be a weighted compact subset
and $\mu$ be a probability measure on $K$. The functional $\log\vol\cB^{2}(\mu,\phi)$
is convex in its $\mu$-variable and concave in its $\phi$-variable. \end{prop}
\begin{proof}
The function $-\log\det$, defined on $\Herm^{+}(N)$, is convex for
its linear structure. Since the map $\mu\mapsto L^{2}(\mu,\phi)$
sending $\mu$ to the corresponding Gram matrix is clearly affine,
formula (\ref{equ:L-func}) implies that \[
\mu\mapsto\log\vol\cB^{2}(\mu,\phi)\]
 is convex on the space of positive measures. Concavity in $\phi$
follows from (\ref{equ:formule1}) and H{\"o}lder's inequality. 
\end{proof}

\subsection{Directional derivatives}

\label{sec:direc}

\begin{prop}
\label{prop:derivatives} The $\cL$-functional has directional derivatives
given by \[
\frac{\partial}{\partial\phi}\log\vol\cB^{2}(\mu,\phi)=\langle N\beta(\mu,\phi),\cdot\rangle\]
 and \[
\frac{\partial}{\partial\mu}\log\vol\cB^{2}(\mu,\phi)=-\langle\cdot,\rho(\mu,\phi)\rangle.\]
\end{prop}

\begin{proof}
This is very similar to Lemma 6.4 in \cite{BB08a}, itself a variant
of Lemma 2 of \cite{Don05a}. By (\ref{equ:L-func}) we have to show
that given two paths $\phi_{t}$, $\mu_{t}$ we have \[
\frac{d}{dt}_{t=0}\log\det\left(\int_{X}s_{i}\overline{s}_{j}e^{-2\phi_{t}}d\mu\right)_{i,j}=-2\int_{X}\left(\frac{d}{dt}_{t=0}\phi_{t}\right)\rho(\mu,\phi_{0})d\mu\]
 and \[
\frac{d}{dt}_{t=0}\log\det\left(\int_{X}s_{i}\overline{s}_{j}e^{-2\phi}d\mu_{t}\right)_{i,j}=\int_{X}\rho(\mu_{0},\phi)\left(\frac{d}{dt}_{t=0}d\mu_{t}\right).\]
 The only thing to remark is that the variations are independent of
the choice of the basis $S$ (see \cite{BB08a}), so that one can assume that $S=(s_{j})$
is an orthonormal basis for $L^{2}(\mu,\phi)$. The result then follows
from a straightforward computation. 
\end{proof}

If $(\mu,\phi)$ is a weighted subset, the condition 
\[
\beta(\mu,\phi)=\mu
\]
holds by definition iff 
\[
\rho(\mu,\phi)=N\,\,\mu\mathrm{-a.e}.
\]
According to Proposition \ref{prop:derivatives}, this is the case
iff $\phi$ is a critical point of the convex functional 
\[
N\mu-\log\vol\cB^{2}(\mu,\cdot).
\]
On the other hand this condition is related to Donaldson's notion
of \emph{$\mu$-balanced metric} (cf. \cite{Don05b}, Section 2.2).
Indeed $\phi$ is $\mu$-balanced in Donaldson's sense iff $\rho(\mu,\phi)=N$
holds everywhere on $X$.

\begin{prop}
\label{prop:balanced} For any configuration $P\in X^{N}$, the pair
$(\delta_{P},\phi)$ satisfies \[
\beta(\delta_{P},\phi)=\delta_{P}.\]
\end{prop}

\begin{proof}
This follows for example by differentiating (\ref{equ:formule2})
with respect to $\phi$, using Proposition \ref{prop:derivatives}
and the fact that $\delta_{P}$ is the derivative with respect to
$\phi$ of \[
-\frac{1}{N}\log|\det S|_{\phi}(P)=\frac{1}{N}\sum_{j}\phi(x_{j})-\frac{1}{N}\log|\det S|(P).\]
\end{proof}

On the other hand, following \cite{BBLW08} we introduce \begin{defi} If \((K,\phi)\) is a weighted compact subset, we say that a probability measure \(\mu\) on \(K\) is a \((K,\phi)\)-\emph{optimal measure} iff it realizes the minimum of \(\log\vol\cB^2(\cdot,\phi)\) over the compact convex set \(\cP_K\) of all probability measures on \(K\).
\end{defi} As in \cite{Bos90}, one shows:
\begin{prop}
\label{prop:optimal} A probability measure $\mu$ on $K$ is $(K,\phi)$-optimal
iff \[
\sup_{K}\rho(\mu,\phi)=N.\]
 In particular we then have \[
\beta(\mu,\phi)=\mu\]
. \end{prop}
\begin{proof}
By convexity of $\mu\mapsto\log\vol\cB^{2}(\cdot,\phi)$, $\mu$ realizes
its minimum on $\cP_{K}$ iff \[
\langle\frac{\partial}{\partial\mu}\log\vol\cB^{2}(\phi,\mu),\nu-\mu\rangle\geq0\]
 for all $\nu\in\cP_{K}$, i.e.~iff \[
\langle\rho(\phi,\mu),\nu\rangle\leq N\]
 for all probability measures $\nu$ on $K$, which is in turn equivalent
to \[
\sup_{K}\rho(\phi,\mu)=N\]
 and implies $\rho(\mu,\phi)=N$ $\mu$-a.e.~since $\langle\rho(\mu,\phi),\mu\rangle=N$. 
\end{proof}
We note that the optimal value satisfies \[
\min_{\mu\in\cP_{K}}\log\vol\cB^{2}(\mu,\phi)\geq\log\vol\cB^{\infty}(K,\phi),\]
 but equality does \emph{not} hold as soon as $N\geq2$ since it would
imply that $\cB^{\infty}(K,\phi)=\cB^{2}(\mu,\phi)$ for some measure
$\mu\in\cP_{K}$ and thus that $1=\sup_{K}\rho(\mu,\phi)\geq N$.

Next, we have the following basic
\begin{prop}
\label{pro:dist for fekete}Let $P$ be a Fekete configuration for
the weighted set $(K,\phi).$ Then \[
\sup_{K}\left|s(x)\right|_{\phi}\leq\sum_{x_{i}\in P}\left|s(x_{i})\right|_{\phi}\]
for any $s\in H^{0}(L),$ i.e. the $L^{\infty}(K)-L^{1}(\delta_{P})$
distortion is at most equal to $N.$\end{prop}
\begin{proof}
Fix a configuration $P=(x_{1},...,x_{N})$ and let $e_{i}\in H^{0}(L\otimes L_{x_{i}}^{*})$
be defined by $e_{i}(x):=\det S(x_{1},...,x_{i-1},x,x_{i},...,x_{N})\otimes\det S(x_{1},...,x_{i}...,x_{N})^{-1}$ (the Lagrange interpolation "polynomials"). 
Then any $s\in H^{0}(L)$ may be written as \[
s(x)=\sum_{i=1}^{N}s(x_{i})\otimes e_{i},\]
 using the natural identification between $L^{*}$ and $L^{-1}.$
Hence, \[
\left|s(x)\right|_{\phi}\leq\sum_{i=1}^{N}\left|s(x_{i})\right|_{\phi}\left|e_{i}(x)\right|_{\phi}.\]
 Finally, if $P$ is a Fekete configuration for $(K,\phi),$ then
clearly $\left|e_{i}(x_{i})\right|_{\phi}\leq1,$ which finishes the
proof of the proposition.
\end{proof}

\section{Energy at equilibrium}

As in Section \ref{sub:donald} in the introduction, we now suppose
given a reference weighted compact subset $(K_{0},\phi_{0})$. We
normalize the Haar measure $\vol$ on $H^{0}(kL)$ by the condition
\[
\vol\cB^{\infty}(K_{0},k\phi_{0})=1\]
 and we consider the corresponding $\cL$-functionals. In other words
we set \[
\cL_{k}(\mu,\phi)=\frac{1}{2kN_{k}}\log\frac{\vol\cB^{2}(\mu,k\phi)}{\vol\cB^{\infty}(K_{0},k\phi_{0})}\]
 and \[
\cL_{k}(K,\phi)=\frac{1}{2kN_{k}}\log\frac{\vol\cB^{\infty}(K,k\phi)}{\vol\cB^{\infty}(K_{0},k\phi_{0})}\]

We will use the following results from \cite{BB08a}.
\begin{thm}
\label{thm:transfinite} If $(K,\phi)$ is a given compact weighted
subset, then \[
\lim_{k\rightarrow\infty}\cL_{k}(K,\phi)=\eneq(K,\phi).\]

\label{thm:diff} The map $\phi\mapsto\eneq(K,\phi)$, defined on
the affine space of continuous weights over $K$, is concave and differentiable,
with directional derivatives given by integration against the equilibrium
measure: \[
\frac{d}{dt}_{t=0}\eneq(\phi+tv)=\langle v,\eq(K,\phi)\rangle.\]

\end{thm}
This differentiability property of the energy at equilibrium really
is the key to the proof of Theorem C. Even though $\eneq(K,\phi)$ is
by definition the composition of the projection operator $P_{K}:\phi\mapsto\phi_{K}^{*}$
on the convex set of psh weights with the Aubin-Mabuchi energy $\cE$,
whose derivative at $\phi_{K}^{*}$ is equal to $\eq(K,\phi)$, this
result is not a mere application of the chain rule, since $P_{K}$
is definitely \emph{not} differentiable in general.

\section{Proof of the main results}

\subsection{Proof of Theorem C}

\label{sec:thmC} Let $v\in C^{0}(X)$, and set \[
f_{k}(t):=\cL_{k}(\mu_{k},\phi+tv)\]
 and \[
g(t):=\eneq(K,\phi+tv).\]
 Theorem \ref{thm:transfinite} combined with $\cL_{k}(\mu,\phi)\geq\cL_{k}(K,\phi)$
shows that $g(t)$ is an asymptotic lower bound for $f_{k}(t)$ as
$k\rightarrow\infty$, that is \[
\liminf_{k\rightarrow\infty}f_{k}(t)\geq g(t),\]
 and the assumption means that this asymptotic lower bound is achieved
for $t=0$, that is \[
\lim_{k\rightarrow\infty}f_{k}(0)=g(0).\]
 Now $f_{k}$ is concave for each $k$ by Proposition \ref{prop:concave},
and we have \[
f_{k}'(0)=\langle\beta(\mu_{k},k\phi),v\rangle\]
 by Proposition \ref{prop:derivatives}. On the other hand $g$ is
differentiable with \[
g'(0)=\langle\eq(K,\phi),v\rangle\]
 by Theorem \ref{thm:diff}. The elementary lemma below thus shows
that \[
\lim_{k\rightarrow\infty}\langle\beta(\mu_{k},k\phi),v\rangle=\langle\eq(K,\phi),v\rangle\]
 for each continuous function $v$, and the proof of Theorem C is
complete.
\begin{lem}
\label{lem:elementary} Let $f_{k}$ by a sequence of \emph{concave}
functions on $\R$ and let $g$ be a function on $\R$ such that 
\begin{itemize}
\item $\liminf_{k\rightarrow\infty}f_{k}\geq g$. 
\item $\lim_{k\rightarrow\infty}f_{k}(0)=g(0)$. 
\end{itemize}
If the $f_{k}$ and $g$ are differentiable at $0$, then \[
\lim_{k\rightarrow\infty}f_{k}'(0)=g'(0).\]
\end{lem}
\begin{proof}
Since $f_{k}$ is concave, we have \[
f_{k}(0)+f_{k}'(0)t\geq f_{k}(t)\]
 for all $t$ hence \[
\liminf_{k\rightarrow\infty}tf_{k}'(0)\geq g(t)-g(0).\]
 The result now follows by first letting $t>0$ and then $t<0$ tend
to $0$. 
\end{proof}
The same lemma underlies the proof of Yuan's equidistribution theorem
given in \cite{BB08a}, and was in fact inspired by the variational
principle in the original equidistribution result (in the strictly
psh case) by Szpiro, Ullmo and Zhang \cite{SUZ97}.

\subsection{Proof of Theorem B}

As noted in the introduction, the condition on the sequence of probability
measures $\mu_{k}$ in Theorem C is equivalent to \begin{equation}
\log\frac{\vol\cB^{2}(\mu_{k},k\phi)}{\vol\cB^{\infty}(K,k\phi)}=o(kN_{k}).\label{equ:weakBM}\end{equation}
 This condition can be understood as a weak Bernstein-Markov condition
for the sequence $(\mu_{k})$, in view of the following easy result.
\begin{lem}
\label{lem:vol} For any probability measure $\mu$ on $K$,
\[
0\leq\log\frac{\vol\cB^{2}(\mu,\phi)}{\vol\cB^{\infty}(K,\phi)}\leq N\log\sup_{K}\rho(\mu,\phi).\]

\end{lem}
The proof is immediate if we recall that $\sup_{K}\rho(\mu,k\phi)^{1/2}$
is the distortion between the two norms and $\vol$ is homogeneous
of degree $2N_{k}=\dim_{\R}H^{0}(kL)$.

Since a given measure $\mu$ is Bernstein-Markov for $(K,\phi)$ iff
\[
\log\sup_{K}\rho(\mu,k\phi)=o(k),\]
 we now see that Theorem B directly follows from Theorem C.

\subsection{Proof of Theorem A}

Let $P_{k}\in K^{N_{k}}$ be a Fekete configuration for $(K,k\phi)$.
Since $\b(\delta_{P_k},k\phi_{k})=\de_{P_k}$ by Proposition \ref{prop:balanced},
Theorem C will imply Theorem A if we can show that

\begin{equation}
\lim_{k\rightarrow\infty}\cL_{k}(\delta_{P_{k}},k\phi)=\eneq(K,\phi).\label{equ:Fekemin}\end{equation}

This condition is independent of the choice of the reference weighted
subset $(E_{0},\phi_{0})$ since it is equivalent to (\ref{equ:weakBM})
above. We can thus assume that $(E_{0},\phi_{0})$ admits a Bernstein-Markov
measure, that we denote by $\mu_{0}$.

Now let $S_{k}$ be an orthonormal basis of $H^{0}(kL)$ wrt the reference
Hermitian metric $L^{2}(\mu_{0},k\phi_{0})$. The metric $|\det S_{k}|$
does not depend on the specific choice of an orthonormal basis $S_{k}$,
simply because $|\det U|=1$ for any unitary matrix $U$. We recall
the following definition from \cite{BB08a}, which is a generalization
of Leja and Zaharjuta's notion of \emph{transfinite diameter}. 

\begin{defi} Let \((K,\phi)\) be a weighted compact subset. Its \(k\)-\emph{diameter} (with respect to \((\mu_0,\phi_0)\)) is defined by 
\[\cD_k(K,\phi):=-\frac{1}{kN_k}\log\Vert\det S_k\Vert_{L^\infty(K,k\phi)}
=\inf_{P\in K^{N_k}}\frac{1}{kN_k}\log|\det S_k(P_k)|_{k\phi}^{-1}.\]
\end{defi} A Fekete configuration $P_{k}\in K^{N_{k}}$ for $(K,k\phi)$ is
thus a point $P_{k}\in K^{N_{k}}$ where the infimum defining $\cD_{k}(K,\phi)$
is achieved.

The following result was proved in \cite{BB08a}.
\begin{thm}
If $(K,\phi)$ is a weighted compact subset, then \[
\lim_{k\rightarrow\infty}\cD_{k}(K,\phi)=\eneq(K,\phi).\]

\end{thm}
We set $\mu_{k}:=\delta_{P_{k}}$. Since $P_{k}$ is a Fekete configuration
for $(K,k\phi)$, we have \[
-\frac{1}{kN_{k}}\log|\det S_{k}|_{k\phi}(P_{k})=\cD_{k}(K,\phi)\]
 by definition, and formula (\ref{equ:formule3}) thus implies \[
\frac{1}{kN_{k}}\log\frac{\vol\cB^{2}(\mu_{k},k\phi)}{\vol\cB^{2}(\mu_{0},k\phi_{0})}=\cD_{k}(K,\phi)+\frac{1}{2k}\log N_{k}.\]
 This implies that \[
\cL_{k}(\mu_{k},k\phi)=\frac{1}{2kN_{k}}\log\frac{\vol\cB^{2}(\mu_{k},k\phi)}{\vol\cB^{\infty}(K_{0},k\phi_{0})}\]
 converges to $\eneq(K,\phi)$ as desired, since \[
\log N_{k}=O(\log k)\]
 on the one hand and \[
\log\frac{\vol\cB^{2}(\mu_{o},k\phi_{0})}{\vol\cB^{\infty}(K_{0},k\phi_{0})}=o(kN_{k})\]
 by Lemma \ref{lem:vol} below since $\mu_{0}$ is Bernstein-Markov
for $(K_{0},\phi_{0})$. The proof of Theorem A is thus complete.

\subsection{Proof of Corollary C}

Let us first consider the case $(p,q)=(\infty,\infty),$ i.e. we assume
that the sequence of configurations $P_{k}$ has sub-exponential $L^{\infty}(K)-L^{\infty}(\delta_{P_{k}})$
distortion (formula \ref{eq:dist}) that we denote by $C_{k}.$ Applying (\ref{eq:dist}) successively to each variable the section $\det S_{k}$ (as in~\cite{BB08a} P.30)
and using the fact that $\det S_{k}$ is anti-symmetric yields \[
\left\Vert \det S_{k}\right\Vert _{L^{\infty}(K^{N_{k}},k\phi)}\leq(C_{k})^{N_{k}}|\det S_{k}|(P_{k}).\]
 Since, by assumption $C_{k}=O(e^{\epsilon k})$ for any $\epsilon>0$
it hence follows that the sequence $(P_{k})$ is asymptotically Fekete
for $(K,\phi),$ i.e. the measures $\mu_{k}=\delta_{P_{k}}$ satisfy
the growth conditions in Theorem C, proving the convergence in this
case.

Now consider the case of general pairs $(p,q).$ By the BM-property
of $\mu$ we have \[
\left\Vert s\right\Vert _{L^{\infty}(K)}\leq C_{\epsilon}e^{\epsilon k}\left\Vert s\right\Vert _{L^{p}(\mu,k\phi)}\leq C'_{\epsilon}e^{\epsilon'k}\left\Vert s\right\Vert _{L^{q}(\delta_{P_{k}},k\phi)},\]
 also using the assumption that $P_{k}$ has sub-exponential $L^{p}(K,\mu)-L^{q}(\delta_{P_{k}})$
distortion in the last inequality. Applying Jensen's inequality to
replace the latter $L^{q}-$norm with the correspondng $L^{\infty}$-norm then shows that $\delta_{P_{k}}$ has sub-exponential $L^{\infty}(K)-L^{\infty}(\delta_{P_{k}})$
distortion. But then the convergence follows from the first case considered
above.

Finally, by Proposition \ref{pro:dist for fekete} and H{\"o}lder's inequality
(applied twice) any Fekete sequence $\delta_{P_{Fek,k}}$ has $L^{p}(K,\mu)-L^{q}(\delta_{P_{Fek,k}})$
distortion at most $N_{k}=O(k^{n}).$ In particular, any sequence
$(P_{k})$ which \emph{minimizes} the latter distortion for each $k$
has sub-exponential such distortion. Hence, the convergence in the
optimal cases follows from the case considered above.

\subsection{Proof of Corollary D}

The sections $s_{1},...,s_{N}$ appearing in the construction of the
recursively extremal configuration $P=(x_{1},,,,x_{N})$ constitute
an orthononormal basis $S$ in $H^{0}(L).$ Moreover, by definition,
$x_{j}$ maximizes the Bergman distortion function $\rho^{\mathcal{H}_{j}}(x)$
of the sub-Hilbert space $\mathcal{H}_{j}$ and \[
(i)\,\rho^{\mathcal{H}_{j}}(x_{j})=\left|s_{j}(x_{j})\right|_{\phi}^{2},\,\,\,(ii)\, s_{i}(x_{j})=0,\, i<j\]
 Indeed, $(i)$ is a direct consequence of the extremal definition
\ref{equ:distortion} of the Bergman distortion function $\rho^{\mathcal{H}_{j}}$
of the space $\mathcal{H}_{j}.$ Then $(ii)$ follows from $(i)$
by expanding $\rho^{\mathcal{H}_{j}}$ in terms of the orthonormal
base $s_{1},...,s_{j}$ of $\mathcal{H}_{j}$ (using formula \ref{eq:rho in base})
and evaluating at $x_{j}.$

Now, by $(ii)$ above we have that the matrix $(s_{i}(x_{j}))$ is
triangular and hence\[
(\det S)(P):=\det(s_{i}(x_{j}))=s_{1}(x_{1})\cdots s_{N}(x_{N})\]
 Hence, $(i)$ gives that\[
\left|(\det S)(P)\right|_{\phi}^{2}=\rho^{\mathcal{H}_{1}}(x_{1})\cdots\rho^{\mathcal{H}_{N}}(x_{N})\]
 But since $x_{i}$ maximizes $\rho^{\mathcal{H}_{i}}(x)$ where $\int_{X}\rho^{\mathcal{H}_{i}}(x)d\mu=\dim\mathcal{H}_{i}=i$
it follows that $\rho^{\mathcal{H}_{i}}(x)\geq i.$ Thus, $\left|(\det S)(P)\right|_{\phi}^{2}/N!\geq1$
and replacing $P$ by $P_{k}$ then gives that $P_{k}$ is asymptotically
Fekete, i.e. \ref{eq:asymptotically fekete} holds. The corollary
now follows from Theorem $C.$

\end{document}